\documentclass[10pt,a4paper]{amsart}
\usepackage{amssymb}
\usepackage{amsmath}
\usepackage{amsfonts,a4wide}
\usepackage{bbm}
\usepackage{enumerate}
\usepackage[latin1]{inputenc}
\usepackage{shadow}
\usepackage{slashed}
\usepackage[all]{xy}
\usepackage{tikz-cd}
\usepackage[letterpaper, margin=3cm]{geometry}
\usepackage{hyperref}
\usepackage{cleveref}
\usepackage{thmtools, thm-restate}

\newcommand\QQ{\mathbb{Q}}
\newcommand\CC{\mathbb{C}}
\newcommand\PP{\mathbb{P}}
\newcommand\RR{\mathbb{R}}
\numberwithin{equation}{section}

\theoremstyle{definition}
\newtheorem{theorem}[subsection]{Theorem}

\newtheorem{proposition}[subsection]{Proposition}

\newtheorem*{fact}{Fact}
\newtheorem{corollary}[subsection]{Corollary}
\newtheorem{lemma}[subsection]{Lemma}
\newtheorem*{conjecture}{Conjecture}
\newtheorem{remark}[subsection]{Remark}

\DeclareMathOperator{\Td}{Td}

\title{Almost complex manifolds with total Betti number three}
\author{Jiahao Hu}
\address{Stony Brook University, Department of Mathematics, 100 Nicolls Road, 11794 Stony Brook}
\email{jiahao.hu@stonybrook.edu}

\begin{document}

\maketitle

\begin{abstract}
    We show the minimal total Betti number of a closed almost complex manifold of dimension $2n\ge 8$ is four, thus confirming a conjecture of Sullivan except for dimension $6$. Along the way, we prove the only simply connected closed complex manifold having total Betti number three is the complex projective plane.
\end{abstract}

\section{Introduction}
The group $Sp(2n,\RR)$ of symplectic transformations and the group $GL(n,\CC)$ of complex linear transformations of a $2n$--dimensional real vector space share a common maximal compact subgroup---the unitary group $U(n)$. Consequently smooth manifolds admitting \textit{almost symplectic} structures are the same as those admitting \textit{almost complex} structures. An almost symplectic structure is a non-degenerate two form $\omega$, while an almost complex structure is an automorphism $J$ of the tangent bundle so that $J^{2}=-id$. 
With any given Riemannian metric, $\omega$ and $J$ generates one another in view of $Sp(2n,\RR)\cap O(2n)=GL(n,\CC)\cap O(2n)=U(n)$. Two different integrability conditions are usually imposed to remove the word \textit{almost}. They are: the \textit{symplectic} condition $d\omega=0$ yielding a local Darboux chart, or the \textit{complex} condition that the Nijenhaus tensor $N_J=0$ yielding a local holomorphic chart.

The symplectic condition is directly reflected in the topology of closed manifolds: the de Rham cohomology class of the symplectic form $\omega$ generates a subring that is isomorphic to the de Rham cohomology of $\CC\PP^n$. This implies:
\begin{fact}
	The minimal total Betti number (i.e. the sum of all Betti numbers) of a $2n$--dimensional closed \textit{symplectic} manifold is $n+1$.
\end{fact}
In contrast, however, the topological consequence of the complex condition is widely unknown, except in dimension two and four, featuring the work of Riemann and Kodaira. Motivated by that the complex realm \textit{should} be mirror to the symplectic realm, Dennis Sullivan proposes to study the minimal total Betti numbers of complex manifolds with respect to their dimensions. In mirror correspondence to the above fact, he formulated:
\begin{conjecture}[Sullivan]
\footnote{According to Sullivan, he was inspired by how information is beautifully organized through graphs in fluid dynamics, e.g. K41, and observing the minimal total Betti numbers of symplectic manifolds grow linearly with respect to dimension, fitting into a nice linear graph. This conjecture thus gives the corresponding graph for complex manifolds a particularly nice form: first grows linearly and then flats out. See also \cite[Fig. 1]{AM19}.}
The minimal total Betti number of a $2n$--dimensional closed \textit{complex} manifold is $\min\{n+1,4\}$.
\end{conjecture}

The object of this paper is to confirm this conjecture except when $n$ is $3$. We note total Betti number $4$ is achieved by Hopf and Calabi-Eckmann manifolds in all even dimensions. Also a confirmation for $n=3$ will imply $S^6$ admits no complex structure.

It turns out quite surprisingly, even though the conjecture is made for \textit{complex} manifolds, Albanese and Milivojevi\'c \cite{AM19} proved this statement in fact holds in the \textit{almost complex} category when $n$ is neither $3$ nor of the form $2^l$ for $l\ge 10$. Our result below settles the latter case they left out.

\begin{restatable}{theorem}{minimalbetti}\label{minimalbetti}
Let $M$ be a closed \textit{almost complex} manifold of dimension $2n\ge 8$. Then the total Betti number of $M$ is $\ge 4$.
\end{restatable}

By Poincar\'e duality, the total Betti number of a (nonempty positive dimensional) closed orientable manifold is at least two. In \cite{AM19} it is shown (positive dimensional) almost complex manifolds having total Betti number exactly two must live in dimension $2$ and $6$. \Cref{minimalbetti} is therefore achieved by studying almost complex manifolds with total Betti number $3$. Our main theorem in this direction is:
\begin{restatable}{theorem}{cobordism}\label{cobordism}
Let $M$ be a positive dimensional closed \textit{almost complex} manifold whose total Betti number is $3$. Then $\dim M=4$ and $M$ is complex cobordant to $\CC\PP^2$.
\end{restatable}

As consequences of \Cref{cobordism}, we immediately obtain:
\begin{theorem}\label{homeomorphism}
Let $M$ be a simply connected closed \textit{almost complex} manifold whose total Betti number is $3$. Then $M$ is homeomorphic to $\CC\PP^2$.
\end{theorem}

\begin{theorem}\label{biholomorphism}
Let $M$ be a simply connected closed \textit{complex} manifold whose total Betti number is $3$. Then $M$ is biholomorphic to $\CC\PP^2$.
\end{theorem}

\begin{proof}[Proof of \Cref{homeomorphism}]
The only zero--dimensional simply connected manifold is $\RR^0$, whose total Betti number is one, hence $\dim M>0$. Then by \Cref{cobordism}, $M$ is a $4$--manifold. Since $M$ is now assumed to be simply connected, $H_1(M;\mathbb{Z})=0$. Thus by the universal coefficient theorem, $H^2(M;\mathbb{Z})$ is free. Furthermore by Poincar\'e duality and the total Betti number three assumption, $H^2(M;\mathbb{Z})$ is of rank one. So the intersection form of $M$ is determined by its signature, which is equal to one by \Cref{cobordism} (because signature is a cobordism invariant). Our theorem then follows from Freedman's theorem \cite[Theorem 1.5]{Freedman} that the homeomorphism type of a simply connected closed smooth $4$--manifold is determined by its intersection form.
\end{proof}

\begin{proof}[Proof of \Cref{biholomorphism}]
This follows from \Cref{homeomorphism} and Yau's theorem \cite[Theorem 5]{Yau} that any compact complex surface homotopy equivalent to $\CC\PP^2$ is biholomorphic to $\CC\PP^2$.
\end{proof}

It should be noted the simply-connectedness assumption is essential to both \Cref{homeomorphism} and \Cref{biholomorphism}. There do exist non-simply-connected complex surfaces, so called fake projective planes (see e.g. \cite{fakeproj}), having the same Betti numbers as $\CC\PP^2$.

The heart of our proof of \Cref{cobordism} lies in the integrality of the signature and (especially) the Todd genus of an almost complex manifold. Another important ingredient is, on an almost complex manifold the Todd class coincides with the $\hat{A}$ class up to the exponential of half the first Chern class (see \cite[pp. 197]{Hirzebruch}):
\begin{equation}\label{a=td}
    \hat{A}=e^{c_1/2} \Td.
\end{equation}
In particular if all the Chern numbers that involves the first Chern class are zero, then the Todd genus is equal to the $\hat{A}$ genus. Our major technical tool is therefore the intimate relation between the signature and the $\hat{A}$ genus, which we will establish in \Cref{sec2}. In \Cref{sec3} we prove \Cref{cobordism} and derive \Cref{minimalbetti} as an easy corollary.

\section*{Acknowledgements}
\Cref{minimalbetti} was also independently obtained by Zhixu Su and was announced by her in a conference in 2018; her argument differs from ours (see \Cref{Su} for a sketch of her proof). The author would like to thank Su for informing him on this matter and for sharing her proof. The author also wishes to thank Dennis Sullivan and Aleksandar Milivojevi\'c for their encouragement, helpful discussions and valuable comments on early versions of this paper.

\section{Relation between the signature and $\hat{A}$ genus}\label{sec2}
Let us introduce our notation. Throughout $M$ is a nonempty closed oriented smooth manifold of dimension $\dim M=2n>0$. The total Betti number of $M$ is the sum of all Betti numbers of $M$, namely $\sum_{i\ge 0}\dim H^i(M;\QQ)$. By $p_i$ we mean the $i^\text{th}$ Pontryagin class of $M$. If $M$ is further almost complex, $c_i$ will be its $i^\text{th}$ Chern class. The symbol $\int_M$ means pairing a cohomology class with the fundamental class of $M$. The signature, Euler characteristic, $\hat{A}$ genus and Todd genus (if defined) of $M$ will be denoted by $\sigma(M)$, $\chi(M)$, $\hat{A}(M)$ and $\Td(M)$ respectively.

Since the appearance of Hirzebruch's celebrated work on multiplicative sequences and his signature theorem, $\sigma(M)$ and $\hat{A}(M)$ are known to be rational linear combinations of Pontryagin numbers. That is,

\begin{equation}\label{singnaturethm}
    \sigma(M)=\sum h_{i_1,\dots, i_r}\int_M p_{i_1}\cdots p_{i_r}, \quad\hat{A}(M)=\sum a_{i_1,\dots, i_r}\int_M p_{i_1}\cdots p_{i_r}.
\end{equation}

The coefficients of $\int_M p_m$ can be obtained by applying the Cauchy formula to the characteristic power series associated to the signature and $\hat{A}$ genus, which are $Q_\sigma(z)=\frac{\sqrt{z}}{\tanh \sqrt{z}}$ and $Q_{\hat{A}}(z)=\frac{\sqrt{z}/2}{\sinh(\sqrt{z}/2)}$ respectively (see \cite[pp. 16,19]{HBJ}, \cite[pp.9-11]{Hirzebruch}). The results are well-known (cf. \cite[pp.  12-13]{Hirzebruch}):
\begin{equation*}
    h_m=\frac{2^{2m}(2^{2m-1}-1)}{(2m)!}B_m,\quad a_m=\frac{-B_m}{2(2m)!}.
\end{equation*}
Here
\begin{equation}\label{bernoullidefn}
    B_m=\frac{(2m)!\cdot \zeta(2m)}{2^{2m-1}\pi^{2m}}
\end{equation}
is the $m^\text{th}$ nontrivial Bernoulli number without sign, $\zeta$ is the Riemann zeta function (see \cite[pp. 129-131]{HBJ}). For instance
$$B_1=\frac{1}{6}, h_1=\frac{1}{3}, a_1=-\frac{1}{24}; B_2=\frac{1}{30}, h_2=\frac{7}{45}, a_2=-\frac{1}{1440}.$$ Clearly $h_m$ and $a_m$ are related by
\begin{equation}\label{leadingcoeff}
    h_m=-2^{2m+1}(2^{2m-1}-1)a_m.
\end{equation}
This immediately proves the following well-known lemma (cf. \cite[pp. 90]{HBJ}).
\begin{lemma}\label{lem1}
Let $M$ be a manifold of dimension $4k$ with $\int_M p_k$ being the only possibly nonzero Pontryagin number. Then
\[
\sigma(M)+2^{2k+1}(2^{2k-1}-1)\hat{A}(M)=0.
\]
\end{lemma}

We will need a generalization of \Cref{lem1} in the following situation.
\begin{lemma}\label{lem2}
Let $M$ be a manifold of dimension $8k$ with $\int_M p_k^2$ and $\int_M p_{2k}$ being the only possibly nonzero Pontryagin numbers. Then
\[
\sigma(M)+2^{4k+1}(2^{4k-1}-1)\hat{A}(M)=2^{4k}(2^{2k}-1)^2\Big(\frac{B_k}{2(2k)!}\Big)^2 \int_M p_k^2.
\]
\end{lemma}
\begin{proof}
From \Cref{singnaturethm} we have
\[
\sigma(M)=h_{2k}\int_M p_{2k}+h_{k,k}\int_M p_k^2,\quad\hat{A}(M)=a_{2k}\int_M p_{2k}+a_{k,k}\int_M p_k^2.
\]
By \cite[Lemma 1.4.1]{Hirzebruch} $h_{k,k}$ and $a_{k,k}$ can be expressed as
\begin{equation}\label{middlecoeff}
    h_{k,k}=\frac{1}{2}h_k^2-\frac{1}{2}h_{2k},\quad a_{k,k}=\frac{1}{2}a_k^2-\frac{1}{2}a_{2k}.
\end{equation}
Therefore we get
\begin{align*}
    \sigma(M)&+2^{4k+1}(2^{4k-1}-1)\hat{A}(M)\\
    &=(h_{2k}+2^{4k+1}(2^{4k-1}-1)a_{2k})\int_M p_{2k}+(h_{k,k}+2^{4k+1}(2^{4k-1}-1)a_{k,k})\int_M p_k^2\\
    &=2^{4k}(2^{2k}-1)^2 a_k^2\int_M p_k^2\\
    &=2^{4k}(2^{2k}-1)^2\Big(\frac{B_k}{2(2k)!}\Big)^2 \int_M p_k^2.
\end{align*}
Here the second to last equality follows from \Cref{leadingcoeff} and \Cref{middlecoeff}.
\end{proof}

Next we turn to almost complex manifolds.
\begin{lemma}\label{lem3}
Let $M$ be an almost complex manifold of dimension $8k$ with $\int_M c_{2k}^2$ and $\int_M c_{4k}$ being the only possibly nonzero Chern numbers. Then
\[
2^{4k+1}\Big[(2^{4k-1}-1)(1-r_k)+(3-2^{2k+1})\Big]\cdot \Td(M)=(1+r_k)\cdot \sigma(M)-2^{4k+1}(2^{2k}-1)^2 \frac{B_{2k}}{(4k)!}\cdot \chi(M),
\]
where $r_k$ is defined by $\binom{4k}{2k} B_k^2\cdot r_k=B_{2k}$.
\end{lemma}
\begin{proof}
Since $M$ is assumed to be almost complex, its top Chern class coincides with its Euler class, therefore $\int_M c_{4k}=\chi(M)$. Also its Pontryagin numbers can be recovered from its Chern numbers. The only possibly nonzero Pontryagin numbers are:
\begin{equation*}
    \int_M p_k^2= \int_M [(-1)^k 2c_{2k}]^2=4\int_M c_{2k}^2, \quad \int_M p_{2k}=\int_M c_{2k}^2+2\int_M c_{4k}.
\end{equation*}
Notice all the Chern numbers involving $c_1$ vanish, so by \Cref{a=td} we have $\hat{A}(M)=\Td(M)$.

Therefore from the proof of \Cref{lem2} we have
\begin{equation}\label{sign&ahat}
\sigma(M)+2^{4k+1}(2^{4k-1}-1)\Td(M)=2^{4k+2}(2^{2k}-1)^2 a_k^2 \int_M c_{2k}^2.
\end{equation}
On the other hand, we also have
\begin{equation}\label{toddinchern}
    \Td(M)=\hat{A}(M)=a_{2k}\int_M p_{2k}+a_{k,k}\int_M p_k^2=(a_{2k}+4a_{k,k})\int_M c_{2k}^2+2a_{2k}\chi(M).
\end{equation}
Combining \Cref{sign&ahat} and \Cref{toddinchern} to eliminate $\int_M c_{2k}^2$, we get
\begin{equation}\label{td-sign-ahat}
    2^{4k+2}(2^{2k}-1)^2 a_k^2[\Td(M)-2a_{2k}\chi(M)]=(a_{2k}+4a_{k,k})[\sigma(M)+2^{4k+1}(2^{4k-1}-1)\Td(M)].
\end{equation}
We note $$\frac{a_{2k}+4a_{k,k}}{2a_k^2}=1+r_k.$$ Then the desired identity is a straightforward consequence of \Cref{td-sign-ahat} by dividing by $2a_k^2$ and reorganizing the terms.
\end{proof}

We end this section by estimating the sizes and the relative sizes of the coefficients appearing in \Cref{lem3}. For simplicity, let us denote
\[
C_{\Td}=2^{4k+1}[(2^{4k-1}-1)(1-r_k)+(3-2^{2k+1})],\quad C_\sigma=1+r_k,\quad C_{\chi}=2^{4k+1}(2^{2k}-1)^2 \frac{B_{2k}}{(4k)!}.
\]
Then the conclusion of \Cref{lem3} reads as
\[
C_{\Td}\cdot \Td(M)=C_\sigma\cdot \sigma(M)-C_{\chi}\cdot \chi(M).
\]
\begin{lemma}\label{estimation}
We have the estimates
\[
1<C_\sigma<\frac{3}{2},\quad 0<C_\chi<8(\frac{2}{\pi})^{4k},\quad C_{\Td}\ge \frac{3}{5}\cdot 2^{8k-3}.
\]
\end{lemma}
\begin{proof}
By \Cref{bernoullidefn} and the definition of $r_k$, we see $r_k=\frac{1}{2}\frac{\zeta(4k)}{\zeta(2k)^2}$. Since \[\zeta(2k)^2=\Big(\sum_{m\ge 1}\frac{1}{m^{2k}}\Big)\Big(\sum_{m\ge 1}\frac{1}{m^{2k}}\Big)>\sum_{m\ge 1}\frac{1}{m^{4k}}=\zeta(4k),\] we have $$1<C_\sigma=1+r_k<\frac{3}{2}.$$ 
By \Cref{bernoullidefn} and $\zeta(4k)<\zeta(2)=\pi^2/6<2$ we get
\[
0<C_{\chi}=4\Big(\frac{2^{2k}-1}{\pi^{2k}}\Big)^2\zeta(4k)<8(\frac{2}{\pi})^{4k}.
\]
Finally, for $k\ge 2$
\[
C_{\Td}>2^{4k+1}[\frac{1}{2}(2^{4k-1}-1)+(3-2^{2k+1})]>2^{4k+1}[2^{4k-2}-2^{2k+1}]\ge 2^{8k-2}>\frac{3}{5}\cdot 2^{8k-3}.
\]
As for $k=1$, one can directly verify $r_1=\frac{1}{5}$ and $C_{\Td}=\frac{3}{5}\cdot 2^5$.

\end{proof}

\begin{corollary}\label{cor1}
We have the further estimates
\[
\frac{C_{\Td}}{C_\sigma+C_\chi}>4^k,\quad\frac{C_\sigma}{C_\chi}\ge 3^{k}.
\]
\end{corollary}
\begin{proof}
By \Cref{estimation}
\[
\frac{C_{\Td}}{C_\sigma+C_\chi}>\frac{3\cdot 2^{8k-3}}{15/2+40(2/\pi)^{4k}}\ge\frac{3\cdot 2^{8k-3}}{15/2+40(2/\pi)^{4}}>\frac{2^{8k-3}}{5}>4^k,
\]
and
\[
\frac{C_\sigma}{C_\chi}>\frac{1}{8}(\frac{\pi}{2})^{4k}\ge 3^k \quad\text{for }k\ge 3.
\]
For $k=1,2$ one can directly verify $\frac{C_\sigma}{C_\chi}=3,15$ respectively, so the desired estimate holds.
\end{proof}

The estimates in \Cref{estimation} and \Cref{cor1} are not sharp. Nevertheless they will be sufficient for our applications.

\section{Proof of main theorems}\label{sec3}
In this section we prove \Cref{minimalbetti} and \Cref{cobordism}. Recall the following theorem of Milnor (cf. \cite{Milnor}).
\begin{theorem}[Milnor]\label{milnor}
Two stably almost complex manifolds are complex cobordant if and only if they have the same Chern numbers.
\end{theorem}

\begin{proposition}\label{bettibig}
If $M$ is an almost complex manifold of dimension $8k$ with $\int_M c_{2k}^2$ and $\int_M c_{4k}$ being the only possibly nonzero Chern numbers. Then either $M$ is complex cobordant to the empty manifold or the total Betti number of $M$ is at least $3^{k}$.
\end{proposition}
\begin{proof}
If $\Td(M)\neq 0$, then since $\Td(M)$ is an integer, by \Cref{lem3} and \Cref{cor1} we get
\[
\text{Total Betti number}\ge \max\{|\sigma(M)|,|\chi(M)|\}>\frac{C_{\Td}}{C_\sigma+C_\chi}|\Td(M)|> 4^{k}.
\]
If $\Td(M)=0$ but $\sigma(M)\neq 0$, then again by \Cref{lem3} and \Cref{cor1} we have
\[
\text{Total Betti number}\ge|\chi(M)|=\frac{C_\sigma}{C_\chi}|\sigma(M)|\ge 3^{k}.
\]
If $\Td(M)=\sigma(M)=0$, then from \Cref{sign&ahat} and \Cref{toddinchern} we have $\int_M c_{2k}^2=0$ and $\int_M c_{4k}=\chi(M)=0$. So all the Chern numbers of $M$ are zero, hence $M$ is complex cobordant to the empty manifold by \Cref{milnor}. 
\end{proof}

We are ready to prove our main theorems.
\cobordism*

\begin{proof}
By \cite[Theorem 2.2, Theorem 3.3, Remark 4.2]{AM19} $\dim M$ is either $4$ or divisible by $2^{11}$. We must show the latter cannot happen. Assume otherwise $\dim M=8k$ for some $k\ge 2$. Since the total Betti number of $M$ is assumed to be three, by Poincar\'e duality the only non-trivial rational cohomology groups are $H^0$, $H^{4k}$ and $H^{8k}$, each of which is one-dimensional. Therefore the only possibly nonzero Chern numbers are $\int_M c_{2k}^2$ and $\int_M c_{4k}$. Then by \Cref{bettibig}, either $M$ is complex cobordant to the empty manifold which contradicts with $\int_M c_{4k}=\chi(M)=3\neq 0$ (if it is complex cobordant to the empty manifold then all its Chern numbers are zero by \Cref{milnor}), or the total Betti number is at least $3^{k}\ge 9$ which contradicts the total Betti number three assumption. This completes the proof of our first assertion.

To prove the second assertion, notice now $\int_M c_2=\chi(M)=3$ and
\[
\sigma(M)=\int_M\frac{p_1}{3}=\int_M\frac{1}{3}(c_1^2-2c_2)=\frac{1}{3}\int_M c_1^2-2.
\]
Meanwhile (cf. \cite[pp. 14]{Hirzebruch}),
\[
\Td(M)=\int_M\frac{1}{12}(c_1^2+c_2)=\frac{1}{12}\int_M c_1^2+\frac{1}{4}.
\]
It follows that $\sigma(M)+3=4\cdot \Td(M)$, so $\sigma(M)\equiv 1$ modulo $4$. Similar as before, by Poincar\'e duality and the total Betti number three assumption, the middle cohomology is one-dimensional. So $\sigma(M)$ must be one. This in turn implies $\int_M c_1^2=9$. We observe the Chern numbers of $M$ are the same as those of $\CC\PP^2$. So by \Cref{milnor} we conclude $M$ is complex cobordant to $\CC\PP^2$.
\end{proof}

\begin{remark}\label{Su}
We here communicate another proof of that $8k$-manifolds with total Betti number three cannot admit almost complex structures, independently obtained by Zhixu Su\footnote{Our proof above was attained in summer 2019; later the author learned from Su that she had also proved this and announced it in a conference in 2018.}. Briefly, in such case the Chern classes $c_1,\dots,c_{2k-1}$ are torsion, then from Hattori-Stong relations one can show $\int_{M} c_{2k}^2$ is divisible by $[(2k-1)!]^2$. This extra divisibility combined with a careful $2$-adic examination of the signature equation gives the result.
\end{remark}

\minimalbetti*

\begin{proof}
This follows from \Cref{cobordism} and \cite[Theorem 2.2]{AM19} that almost complex manifolds with total Betti number two exist only in dimension $2$ and $6$.
\end{proof}

\bibliographystyle{plain}
\bibliography{ref}

\end{document}